\documentclass[10pt]{amsart}
\usepackage{amssymb, mathrsfs, bbm}
\usepackage[margin=1.25in]{geometry}

\usepackage[all]{xy}
\SelectTips{eu}{}
\entrymodifiers={+!!<0pt,\fontdimen22\textfont2>}

\usepackage[pdftex]{hyperref}
\usepackage{url}

\swapnumbers
\theoremstyle{plain}
\newtheorem{thm}[subsection]{Theorem}
\newtheorem{lemma}[subsection]{Lemma}
\newtheorem{prop}[subsection]{Proposition}
\newtheorem{cor}[subsection]{Corollary}

\theoremstyle{remark}

\newtheorem{remark}[subsection]{Remark}

\theoremstyle{definition}

\numberwithin{equation}{subsection}

\def\cD{\mathcal{D}}

\def\cO{\mathcal{O}}

\def\11{\mathbf{1}}

\def\AA{\mathbf{A}} 
 
\def\CC{\mathbf{C}} 
\def\DD{\mathbf{D}}

\def\II{\mathbf{i}}

\def\PP{\mathbf{P}} 
\def\qq{\mathbf{q}}
\def\QQ{\mathbf{Q}}

\def\TT{\mathbf{T}}

\def\ZZ{\mathbf{Z}}

 \def\Coh{\mathrm{Coh}}
 
 \def\conv{\cdot}
  
 \def\DR{\mathscr{DR}}
 \def\End{\mathrm{End}}
 \def\Ext{\mathrm{Ext}}
 
 \def\For{\mathrm{For}}
 \def\gr{\mathrm{gr}}

 \def\Hecke{\mathscr{H}}

 \def\id{\mathrm{id}}

 \def\MHM{\mathrm{MHM}}
 \def\Nilcone{\mathscr{N}}

 \def\pt{\mathrm{pt}}
 \def\qq{\mathbf{q}}
 
 \def\rat{\mathrm{rat}}

 \def\Spec{\mathrm{Spec}}

 \newcommand{\mapright}[1]{\xrightarrow{#1}}
 \newcommand{\mapleft}[1]{\xleftarrow{#1}}

 \newcommand{\const}[1]{\underline{#1}}

 \newcommand{\qtimes}[1]{\mathop{\times}\limits^{#1}}
 \newcommand{\ttimes}{\mathop{\widetilde{\boxtimes}}}

\makeatletter
\renewcommand{\@makefnmark}{\mbox{\textsuperscript{}}}
\makeatother

\title{A remark on braid group actions on coherent sheaves}
\author{R. Virk}
\begin{document}
\maketitle
%\setcounter{tocdepth}{1}
%\tableofcontents
%\renewcommand{\thesubsection}{\textbf{\arabic{section}.\arabic{subsection}}}
\section{Introduction and notation}
Examples of braid group actions on derived categories of coherent sheaves are abundant in the literature. Significant part of the interest in these stems from the relation with homological mirror symmetry (see \cite{ST}).
The purpose of this note is to give a construction of braid group actions on coherent sheaves (algebraic) via actions on derived categories of \emph{constructible sheaves} (topological). 
%The advantage of this approach is that on the constructible side demonstrating the braid relations/invertibility of our functors is relatively easy (compare \S\ref{s:convolution} with \cite[\S4]{KTh}). The downside is the use of some heavy technology: M. Saito's theory of mixed Hodge modules \cite{Sa88}. However, in future work I intend to explain how Hodge modules can be circumvented, see Remark \ref{fairytale}.
%
%The idea is as follows. First, one considers the geometric Hecke algebra (the category $\Hecke$ of \S\ref{s:convolution}; for an explanation of why $\Hecke$ is called the geometric Hecke algebra see \cite{Sp82}. This category is endowed with a monoidal structure (\emph{convolution}, see \S\ref{s:convolution}) and contains a set of invertible objects that satisfy the braid relations (see Proposition \ref{braidrels} and Theorem \ref{invertible}). Next, one constructs monoidal functors from $\Hecke$ to the category of endofunctors of a number of related triangulated categories (see \S\ref{s:classicalhecke}-\S\ref{s:menagerie} and Theorem \ref{gammamonoidal}), hence giving braid group actions on these auxilliary categories. Put another way, one obtains naturally occuring representations of $\Hecke$ (by a \emph{representation} or \emph{action} of $\Hecke$ we mean a monoidal functor from $\Hecke$ to the category of endofunctors of some category).

In \S\ref{s:convolution} we review the construction of the main player (constructible side): $D^b_m(B\backslash G/B)$, the Borel equivariant derived category of mixed Hodge modules on the flag variety $G/B$ associated to a reductive group $G$. The key points here are Prop. \ref{braidrels} (the braid relations) and Thm. \ref{invertible} (invertibility of the objects giving the braid relations). The contents of this section can be found in various forms in the literature, for instance see \cite{Sp82}. 
%The only possible exception to this is Theorem \ref{invertible} whose proof does not seem to have been documented previously. Regardless, even this should be regarded as well known. 
%We end the section with an application to `intertwining functors' (see \S\ref{s:classicalhecke}), and an overview of some naturally occuring examples of $\Hecke$-actions (see \S\ref{s:menagerie}).
%Although this section is written in the language of mixed Hodge modules, the reader may instead prefer to work with $\ell$-adic \'etale sheaves, or even just non-mixed sheaves, or $\cD$-modules. All the statements (modulo weights in non-mixed settings) hold with exactly the same proofs. The special features of mixed Hodge modules only become relevant in the next section.

Underlying a mixed Hodge module $M$ on a smooth variety $X$ is the structure of a filtered $\cD$-module. Taking the associated graded one produces a $\CC^*$-equivariant coherent sheaf $\tilde\gr M$ on the cotangent bundle $T^*X$. This brings us to the main result Thm. \ref{gammamonoidal}, which exploits $\tilde\gr$ to obtain a monoidal functor from $D^b_m(B\backslash G/B)$ to an appropriate category of coherent sheaves $\tilde\Hecke$ on the Steinberg variety. In view of Prop. \ref{braidrels} and Thm. \ref{invertible}, this realizes our goal of obtaining braid group actions on coherent sheaves. Via the standard formalism of Fourier-Mukai kernels, the category $\tilde\Hecke$ of \S\ref{s:coherent} acts on auxilliary categories of coherent sheaves. Hence, one obtains braid group actions on these too.

The idea to exploit $\tilde\gr$ in this fashion comes from T. Tanisaki's beautiful paper \cite{T}. This theme was also explored by I. Grojnowski \cite{Groj}. However, both I. Grojnowski and T. Tanisaki work at the level of Grothendieck groups, we insist on working at the categorical level. 
Regardless, I emphasize that all the key ideas are contained in \cite{T}.
Furthermore, the key technical result (Thm. \ref{laumoncorr}) that is used to prove Thm. \ref{gammamonoidal} is due to G. Laumon \cite{La2}.

A variant of Thm. \ref{gammamonoidal} has also been obtained by R. Bezrukavnikov and S. Riche \cite{BRmhm}. Further, R. Bezrukavnikov and S. Riche were certainly aware of such a result long before this note was conceived (see \cite{BR}). Thus, experts in geometric representation theory have known that such a result must hold for a long time. Certainly V. Ginzburg (see \cite{G}) I. Grojnowski (see \cite{Groj}), M. Kashiwara (see \cite{KaT}), D. Kazhdan (see \cite{KL}), G. Lusztig (see \cite{KL}), R. Rouquier (see \cite{Ro2}), and of course T. Tanisaki (see \cite{T}) must have also known. Undoubtedly this list is woefully incomplete. I request the reader's forgiveness for my ignorance in this matter. 
%Consequently, I claim no great originality for the ideas in this note.
%\section{Notation}\label{s:notation}
\subsection{Conventions regarding varieties}
Throughout `variety' = `separated reduced scheme of finite type over $\Spec(\CC)$'. A variety can and will be identified with its set of geometric points. If $X$ is a variety, set
$d_X = \mathrm{dim_{\CC}}(X)$. If $Y$ is another variety, set $d_{X/Y} = d_X - d_Y$.

\subsection{Mixed Hodge modules}
Write $\MHM(X)$ for the abelian category of mixed Hodge modules on $X$, and $D^b_m(X)$ for its bounded derived category. The constant (mixed Hodge) sheaf in $D^b_m(X)$ is denoted $\const{X}$. The Tate twist is denoted by $(1)$. 
%The non-mixed (i.e., ordinary constructible sheaves) version of $D^b_m(X)$ is denoted $D^b_c(X)$. Write 
%$\rat\colon D^b_m(X)\to D^b_c(X)$
%for the forgetful functor. Then $\rat(\MHM(X))\subseteq \Perv(X)$, where $\Perv(X)\subset D^b_c(X)$ is the abelian subcategory of perverse sheaves (middle perversity).
%
%\subsection{Standard functors}
%With the exception of the tensor product and inner $\Hom$ all functors are assumed to be derived. So we write $\otimes$ (resp. $\RHom$) for the ordinary (underived) tensor product (resp. inner $\Hom$), $\otimes^L$ (resp. $R\RHom$ for their derived functors, $f_*$ for what would normally be denoted $Rf_*$, $\DD$ for Verdier duality, etc.
%
%\subsection{Equivariant derived category}

Let $G$ be a linear algebraic group acting on $X$ (action will always mean left action). Write $D^b_m(G\backslash X)$ for the $G$-equivariant derived category (in the sense of \cite{BL94}) of mixed Hodge modules on $X$. Write $\For\colon D^b_m(G\backslash X) \to D^b_m(X)$ for the forgetful functor. 
\subsection{Coherent sheaves}
For a variety $X$ we write $\cO_X$ for its structure sheaf, $D^b(\cO_X)$ for the (bounded) derived category of $\cO_X$-coherent sheaves and $\Coh(\cO_X)\subset D^b(\cO_X)$ for the abelian subcategory of coherent sheaves. If an algebraic group $G$ acts on $X$, we write $D^G(\cO_X)$ for the (bounded) derived category of $G$-equivariant $\cO_X$-coherent sheaves and $\Coh^G(\cO_X)\subset D^G(\cO_X)$ for the abelian subcategory of $G$-equivariant coherent sheaves.

Given a morphism of varieties $f\colon X\to Y$, when dealing with coherent sheaves, we write $f^{-1}$ for the ordinary pullback of sheaves, so that the pullback
$f^*\colon D^b(\cO_Y)\to D^b(\cO_X)$ is given by 
$f^*M=\cO_Y \otimes_{f^{-1}\cO_X}^L f^{-1}M$.
For a smooth variety $X$ let $\Omega_X$ be the cotangent sheaf on $X$ and set
$\omega_X = \bigwedge^{d_X} \Omega_X$.
%Then Grothendieck-Serre duality
%$\DD\colon D^b(\cO_X)^{\mathrm{op}} \to D^b(\cO_X)$ is given by
%$\DD M = R\RHom(M, \omega_X)[d_X]$.
%The ordinary dual $R\RHom(M,\cO_X)$ is denoted $M^{\vee}$.
If $f\colon X\to Y$ is a morphism of smooth varieties, 
%then the exceptional pullback $f^!\colon D^b(\cO_Y)\to D^b(\cO_X)$ is given by  
%$f^!M=f^*M \otimes_{\cO_X} \omega_{X/Y}[d_{X/Y}]$, where
set $\omega_{X/Y} = \omega_X \otimes_{f^{-1}\cO_Y}f^{-1}\omega_Y^{ -1}$. 
\subsection{$\cD$-modules}
For a smooth variety $X$ we write $\cD_X$ for the sheaf of differential operators on $X$. A $\cD_X$-module will always mean a coherent left $\cD_X$-module that is quasi-coherent as an $\cO_X$-module.
%Our conventions for the standard functors for $\cD$-modules are dictated by requiring that they commute with the Riemann-Hilbert correspondence. In particular,
%Verdier duals
%$\DD\colon D^b(\cD_X) \to D^b(\cD_X)$
%are given by the formula
%$\DD M = R\RHom_{\cD_X}(M, \cD_X\otimes_{\cO_X} \omega_X^{-1})[d_X]$,
%If $f\colon X\to Y$ is a morphism of smooth varieties, then we have a $(\cD_X, f^{-1}\cD_Y)$-bimodule $\cD_{X\to Y}$, and an $(f^{-1}\cD_Y, \cD_X)$-bimodule:
%\[ \cD_{X\to Y} = \cO_X\otimes_{f^{-1}\cO_Y} f^{-1} \cD_Y, \quad \cD_{Y\leftarrow X}= f^{-1}\cD_Y \otimes_{f^{-1}\cO_Y} \omega_{X/Y}.\]
%and $f_*$, $f^!$
%$f_*\colon D^b(\cD_X)\to D^b(\cD_Y)$, $f^!\colon D^b(\cD_Y)\to D^b(\cD_X)$,
%are defined by
%$
%f_*M =f_*(\cD_{Y\leftarrow  X}\otimes^L_{\cD_X} M)$,
%$f^!N =\cD_{X\to Y}\otimes^L_{f^{-1}\cD_Y} f^{-1}N[d_{X/Y}]$,
%where the `$f_*$' on the right hand side of the first formula is the usual (derived) pushforward of sheaves, and $\cD_{X\to Y}$, $\cD_{Y\leftarrow X}$ are the usual transfer bimodules.
%Twisting with $\DD$ we obtain $f_!$ and $f^*$, i.e.,
%$f_!M = \DD f_*(\DD M)$
%and $f^*N = \DD f^*(\DD N)$.
\subsection{Cotangent bundles}\label{ss:sign}
We write $\pi_X\colon T^*X\to X$ for the cotangent bundle to a smooth variety $X$. We identify $T^*(X\times X)$ with $T^*X\times T^*X$. However, we do this via the usual isomorphism $T^*(X\times X)\simeq T^*X\times T^*X$ composed with the antipode map on the right. 
This is dictated by requiring that the conormal bundle to the diagonal in $X\times X$ be identified with the diagonal in $T^*X\times T^*X$.
\section{Convolution}\label{s:convolution}
Let $G$ be a connected reductive group. Fix a Borel subgroup $B\subseteq G$. Then $B$ acts on $G$ via $b\cdot g = gb^{-1}$. The quotient under this action is the flag variety $G/B$. Clearly, $B$ acts on $G/B$ on the left, and we may form the equivariant derived category $D^b_m(B\backslash G/B)$.
%\subsection{Convolution}

Let $\tilde q\colon G \times G/B \to G/B$ be the projection on $G$ composed with the quotient map $G\to G/B$. Let $p\colon G\times G/B\to G/B$ denote projection on $G/B$. Let $q\colon G\times G/B\to G\qtimes{B}G/B$ be the quotient map. Let
$m\colon G\qtimes{B}G/B\to G/B$ denote the map induced by the action of $G$ on $G/B$.
\[ \xymatrixrowsep{1mm}\xymatrix{
& G\times G/B \ar[r]^-{q} \ar@/^0pc/[ld]_-{\tilde{q}} \ar@/^0pc/[rd]^-{p}& G\qtimes{B} G/B \ar[r]^-m & G/B \\
G/B & & G/B
}\]
The \emph{convolution} bifunctor $-\conv - \colon D^b_m(B\backslash G/B)\times D^b_m(B\backslash G/B) \to D^b_m(B\backslash G/B)$ is defined by the formula 
%\begin{equation}\label{convolution}
$M\conv N = m_!(M\ttimes N)$,
%\end{equation}
where $M\ttimes N$ denotes the descent of $\tilde{q}^*M\otimes p^*N[d_B]$ to $D^b_m(B\backslash G\qtimes{B} G/B)$.
This is an associative operation and endows $D^b_m(B\backslash G/B)$ with a monoidal structure.
Convolution adds weights and commutes with Verdier duality, since $m$ is proper.

\subsection{Another description of convolution}\label{ss:altdesc}
The group $G$ acts on $G/B\times G/B$ diagonally, and the map
$G\times G/B \to G/B\times G/B, \quad (g,x) \mapsto (\tilde{q}(g), g\cdot x)$
induces a $G$-equivariant isomorphism $G/B\times G/B \mapright{\sim} G\qtimes{B} G/B$. Under this isomorphism $m\colon G\qtimes{B} G/B \to G/B$ corresponds to projection on the second factor $p_2\colon G/B\times G/B\to G/B$. Define $i\colon G/B \to G/B\times G/B$, $x\mapsto (\tilde{q}(1), x)$. Using equivariant descent (see \cite[Lemma 1.4]{MV}) we infer 
\begin{equation}\label{iso} i^*[-d_{G/B}]\colon D^b_m(G\backslash (G/B\times G/B)) \mapright{\sim} D^b_m(B\backslash G/B) \end{equation}
is a t-exact equivalence. If $M\in D^b_m(G\backslash (G/B\times G/B))$ is pure of weight $n$, then $i^*M[-d_{G/B}]$ is pure of weight $-d_{G/B}$.

Let $r = \id_{G/B}\times \Delta \times \id_{G/B}$,
where $\Delta\colon G/B\to G/B\times G/B$ is the diagonal embedding.
Define a monoidal structure $ - \conv -$ on $D^b_m(G\backslash(G/B\times G/B))$
by the formula
\begin{equation}\label{ordconv} 
M\conv N = p_{13!}r^*(M\boxtimes N)[-d_{G/B}],
\end{equation}
where $p_{13}\colon G/B\times G/B\times G/B \to G/B\times G/B$ denotes projection on the first and third factor.
A diagram chase (omitted) shows that the equivalence \eqref{iso} is monoidal. We will constantly go back and forth between these two descriptions.

\subsection{Braid relations}
Fix a maximal torus $T\subseteq B$. Let $W=N_G(T)/T$ be the Weyl group. Write $\ell\colon W \to \ZZ_{\geq 0}$ for the length function. The $B$-orbits in $G/B$ are indexed by $W$. Further, writing $X_w$ for the orbit corresponding to $w\in W$, we have
$X_w \simeq \AA^{\ell(w)}$.
For each $w\in W$, let $i_w\colon X_w \hookrightarrow G/B$ be the inclusion map. Set
\[
\TT_w = i_{w!}\const{X_w} \quad\mbox{and}\quad
\CC_w = IC(X_w, \const{X_w})[-\ell(w)].\]
Then $\TT_e$ is the unit for convolution and will be denoted by $\11$. Note that both $\TT_w[\ell(w)]$ and $\CC_w[\ell(w)]$ are in $\MHM(B\backslash G/B)$ for all $w\in W$.

Write $Y_w$ for the image of $G\qtimes{B} X_w$ under the isomorphism \eqref{iso}. Then the $Y_w$, $w\in W$, are the $G$-orbits in $G/B\times G/B$.
Furthermore,
$\TT_w = i^*j_{w!}\const{Y_w}$ and $\CC_w=i^*IC(Y_w,\const{Y_w})[-\ell(w)]$, 
where $j_w\colon Y_w\hookrightarrow G/B\times G/B$ is the inclusion map.

\begin{prop}\label{braidrels}
Let $w,w'\in W$. If $\ell(ww')=\ell(w)+\ell(w')$, then
$
\TT_w\conv\TT_{w'} = \TT_{ww'}$ \end{prop}

\begin{proof}
If $\ell(ww')=\ell(w)+\ell(w')$, then $Y_{ww'} = Y_w \times_{G/B} Y_{w'}$, where the fibre product is over the projection maps $Y_w \to G/B$ and $Y_{w'}\to G/B$ on the first and second factor respectively. Now an application of proper base change and the description of convolution on $D^b_m(G\backslash (G/B\times G/B))$ yields the result.
\end{proof}

\begin{lemma}\label{affine}The functor $\TT_w[\ell(w)]\conv -$ is left t-exact, and $(\DD\TT_w)[-\ell(w)]\conv - $ is right t-exact.
\end{lemma}

\begin{proof}It suffices to show $\TT_w[\ell(w)] \conv -$ is left t-exact, since Verdier duality commutes with convolution. Consider the diagram
\[ \xymatrixrowsep{1mm}\xymatrix{
& BwB\times G/B \ar[r]^-{q_w} \ar@/^0pc/[ld]_-{\tilde q_w} \ar@/^0pc/[rd]^-{p_w}& BwB \qtimes{B} G/B \ar[r]^-{m_w} & G/B \\
X_w & & G/B
}\]
where $\tilde q_w$ is the the evident quotient map on the first factor followed by projection, $p_w\colon BwB\times G/B\to G/B$ is projection on the second factor, $q_w$ is the restriction of $q$, and $m_w$ is the restriction of $m$. Then 
$\TT_w \conv - = m_{w!}(\const{X_w} \ttimes - )$, where 
$\const{X_w}\ttimes -$ is the descent of $\tilde q^*_w\const{X_w}\otimes p_w^*(-)[d_B]$ to $D^b_m(B\backslash BwB\qtimes{B} G/B)$.
Now $\const{X_w}[\ell(w)]\ttimes -$ is t-exact. This implies the result, since $m_w$ is affine.
\end{proof}

\begin{prop}\label{convcs}Let $s\in W$ be a simple reflection. Let $G/P_s$ be the partial flag variety corresponding to $s$. Let $\pi_s\colon G/B\to G/P_s$ be the projection map. Then
$\CC_s\conv M = \pi_s^*\pi_{s*} M$, for all $M\in\Hecke$.
\end{prop}

\begin{proof}
The closure $\overline{Y_s}$ of $Y_s$ in $G/B\times G/B$ is smooth (it is isomorphic to $\PP^1\times\PP^1$). Hence, 
$IC(Y_s, \const{Y_s}) = \tilde i_!\const{\overline{Y_s}}[d_{\overline{Y_s}}]$,
where $\tilde i \colon \overline{Y_s}\hookrightarrow G/B\times G/B$ is the inclusion. The map $\pi_s$ is proper. Using proper base change we deduce
$(j_{s!*}\const{Y_s}[d_{Y_s}-1])\conv N = (\id_{G/B}\times \pi_s)^*(\id_{G/B}\times \pi_s)_*N$
for all $N\in D^b_m(G\backslash (G/B\times G/B))$.
Further, if $\tilde M \in D^b_m(G\backslash (G/B\times G/B))$ is such that $i^*\tilde M [-d_{G/B}] = M$, then
$i^*(\id_{G/B}\times \pi_s)^*(\id_{G/B}\times \pi_s)_*\tilde M [-d_{G/B}] = \pi_s^*\pi_{s*} M$.
\end{proof}

\begin{cor}Let $s\in W$ be a simple reflection. Then
$\CC_s \conv \CC_s = \CC_s \oplus \CC_s[-2](-1)$.
\end{cor}
%\begin{proof}The closure $\overline{X_s}$ of $X_s$ in $G/B$ is smooth. In fact, $\overline{X_s}\simeq \PP^1$. Hence,
%\[ \CC_s\conv \CC_s = \pi_s^*\pi_{s!}\CC_s[1] = \CC_s[1] \oplus \CC_s[-1](-1). \qedhere\]
%\end{proof}
Let
$\II\colon D^b_m(B\backslash G/B)\mapright{\sim}D^b_m(B\backslash G/B)$
denote the auto-equivalence induced by the automorphism of $G/B\times G/B$ that switches the factors. Then
$\II (M\conv N) = \II N \conv \II M$,
for all $M,N\in D^b_m(B\backslash G/B)$. Further, if $s\in W$ is a simple reflection, then $\II \TT_s = \TT_s$. Consequently,
$\II \TT_w = \TT_{w^{-1}}$ and $\II \CC_w = \CC_{w^{-1}}$
for all $w\in W$.
\begin{prop}\label{calcprop}Let $s\in W$ be a simple reflection. Then
\begin{enumerate}
\item $\CC_s\conv\TT_s = \CC_s[-2](-1)= \TT_s\conv \CC_s$;
\item $\TT_s \conv \DD\TT_s = \11 =  \DD\TT_s\conv \TT_s$.
\end{enumerate}
\end{prop}

\begin{proof}Prop. \ref{convcs} gives the first equality in (i).
Applying the involution $\II$ gives the second equality in (i). 
Convolve the distinguished triangle $\DD\TT_s\to \11[1] \to \CC_s[3](1)\leadsto$ with $\TT_s$, and use (i) along with Lemma \ref{affine} to get a short exact sequence $0\to \TT_s \conv \DD\TT_s \to \TT_s[1] \to \CC_s[1] \to 0$ in $\MHM(B\backslash G/B)$. This implies $\TT_s\conv \DD\TT_s = \11$. Verdier duality yields $\DD\TT_s \conv \TT_s = \11$.
\end{proof}

\begin{thm}\label{invertible}Each $\TT_w$, $w\in W$, is invertible under convolution.
\end{thm}

\begin{proof}Combine Prop. \ref{braidrels} with Prop. \ref{calcprop}(ii).
\end{proof}
\section{Action on coherent sheaves}\label{s:coherent}
%From here on all functors are assumed to be derived with the following exceptions:
%\begin{quotation}
%\emph{$\otimes$ (resp. $\RHom$) will denote the ordinary (underived) tensor product (resp. inner $\Hom$), the derived tensor product (resp. inner $\Hom$) will be denoted $\otimes^L$ (resp. $R\RHom$).}
%\end{quotation}
Let $X$ be a smooth variety and $\pi_X\colon T^*X\to X$ its cotangent bundle.
\subsection{Filtered $\cD$-modules and mixed Hodge modules}
 Let $F_i(\cD_X)$ denote the sub-sheaf of $\cD_X$ consisting of differential operators of degree at most $i$. This defines a filtration of $\cD_X$. Write $\gr \cD_X$ for the associated graded sheaf of rings. Then we have a canonical isomorphism $\gr \cD_X \simeq \pi_{X*}\cO_{T^*X}$. We identify $\gr\cD_X$ with $\pi_{X*}\cO_{T^*X}$ via this isomorphism.
A \emph{filtered $\cD_X$-module} is a pair $(M,F)$, where $M$ is a $\cD_X$-module and $F$ is an exhaustive filtration of $M$ by sub-sheaves such that
$F_i(\cD_X)F_jM \subseteq F_{i+j}M$. The filtration $F$ is a \emph{good filtration} if 
$\gr(M)$ is coherent as a $\gr\cD_X$-module. The support of $\cO_{T^*X}\otimes_{\pi_X^{-1}\gr\cD_X}\gr(M)$ is the \emph{characteristic variety} of $M$. 

A mixed Hodge module $M\in \MHM(X)$ is a tuple $(M,F, \rat(M), W)$, where $M$ is a regular holonomic $\cD_X$-module, $F$ is a good filtration on $M$ (the \emph{Hodge filtration}), $\rat(M)$ is a perverse sheaf on $X$ with $\QQ$-coefficients (the \emph{rational structure}) such that $\DR(M) = \CC\otimes_{\QQ}\rat(M)$, where $\DR$ is the de Rham functor, and $W$ is the \emph{weight filtration} on $(M,F, \rat(M))$. This data is required to satisfy several compatibilities which we only recall as needed.
Morphisms in $\MHM(X)$ respect the filtrations $F$ and $W$ strictly.
Given $(M, F, \rat(M), W)\in\MHM(X)$,
$M(n) = (M, F_{\bullet - n}, \QQ(n)\otimes_{\QQ} \rat(M), W_{\bullet+2n})$,
where $\QQ(n)=(2\pi\sqrt{-1})^n\QQ$.

The weight filtration $W$ and rational structure $\rat(M)$ are not particularly relevant for us in this section. Consequently, we omit them from our notation from here on and focus on the filtered $\cD$-module structure underlying a mixed Hodge module.

\subsection{The functor $\tilde\gr$}
Let $(M, F)\in \MHM(X)$. Taking the associated graded with respect to $F$ gives a coherent $\gr(\cD_X)$-module $\gr(M)$. Hence, we obtain an exact functor from $\MHM(X)$ to graded coherent $\gr(\cD_X)$-modules.
We have $\CC^{*}$ acting on $T^*X$ via dilation of the fibres of $\pi_X$.
As $\pi_X$ is affine, $\pi_{X*}$ gives an equivalence between $\CC^*$-equivariant quasi-coherent $\cO_{T^*X}$-modules and graded quasi-coherent $\pi_{X*}\cO_{T^*X}$-modules.
Thus, we obtain an exact functor $\tilde\gr\colon \MHM(X)\to \mathrm{Coh}^{\CC^*}(\cO_{T^*X})$. Explicitly,
$\tilde\gr(M) = \cO_{T^*X}\otimes_{\pi_X^{-1}\gr\cD_X} \pi_X^{-1}\gr (M)$,
with $\CC^*$ action on $\tilde\gr(M)$ defined by
$z\cdot (f(x,\xi)\otimes m_i) = f(x,z^{-1}\xi) \otimes z^{-i}m_i$,
where $z\in \CC^*$, $f(x,\xi)\in\cO_{T^*X}$, and $m_i$ is in the $i$-th component of $\gr(M)$. 

\subsection{Tate twist and $\tilde\gr$}
For $n\in \ZZ$ 
let $\qq^n\in \Coh^{\CC^*}(\pt)$ be the one dimensional $\CC^*$-module with the action of $z\in \CC^*$ given by multiplication by $z^n$. Let $a\colon T^*X\to \pt$ be the obvious map. For $M\in D^{\CC^*}(\cO_{T^*X})$, set
$M(n) = a^*\qq^{-n} \otimes_{\cO_{T^*X}} M$.
Evidently, if $N\in D^b_m(X)$, then $\tilde\gr(N(n)) = \tilde\gr(N)(n)$.
\subsection{Correspondences}Let $f\colon X\to Y$ be a morphism of smooth varieties. Associated to $f$ we have the diagram $T^*X\mapleft{f_d} T^*Y\times_Y X \mapright{f_{\pi}} T^*Y$,
where $f_{\pi}$ is the base change of $f$ along $T^*Y\to Y$, and $f_d$ is the map dual to the derivative.
Let $T^*_XX\subseteq T^*X$ denote the zero section. Set $T^*_XY = f_d^{-1}(T^*_XX)$. If $f$ is the inclusion of a closed subvariety, then $T^*_XY$ is the conormal bundle to $X$ in $T^*Y$.
Let $\Lambda\subseteq T^*Y$ be a conic (i.e. $\CC^*$-stable) subvariety. Then $f$ is \emph{non-characteristic} for $\Lambda$ if
$f_{\pi}^{-1}(\Lambda) \cap T^*_XY \subseteq T^*_YY \times_Y X$.
This is equivalent to $f_d|_{f_{\pi}^{-1}(\Lambda)}$ being finite. We say $f$ is non-characteristic for $M\in \MHM(X)$ if $f$ is non-characteristic for the characteristic variety of $M$.

Let $X\mapleft{f}Z\mapright{g}Y$ be a diagram of smooth varieties such that the canonical map $Z\to X\times Y$ is a closed immersion. We call such a diagram a \emph{correspondence} between $X$ and $Y$. Associated to such a correspondence we have a functor $\Phi_{X|Y}\colon D^b_m(X)\to D^b_m(Y)$, $M\mapsto g_*f^*M$. We also have a commutative diagram
\[\xymatrixrowsep{1mm}\xymatrix{
& &T^*_Z(X\times Y)\ar@/^0pc/[ld] \ar@/^0pc/[rd] \ar@/_2pc/[lldd]_-{q_X}\ar@/^2pc/[rrdd]^-{q_Y}\\
& T^*X\times_X Z \ar@/^0pc/[ld]_-{f_{\pi}}\ar@/^0pc/[rd]^-{f_d}& & T^*Y\times_Y Z \ar@/^0pc/[ld]_-{g_d} \ar@/^0pc/[rd]^-{g_{\pi}}\\
T^*X & & T^*Z & & T^*Y
}
\]
with middle square cartesian. So we obtain a correspondence 
$T^*X\mapleft{q_X} T^*_Z(X\times Y) \mapright{q_Y} T^*Y$.
For $M\in D^{\CC^*}(\cO_{T^*X})$ set 
$
\tilde{\Phi}_{X|Y}(M) = q_{Y*}(q_X^*M\otimes_{\cO_{T^*_Z(X\times Y)}} \rho^*\omega_{Z/Y})[d_{Z/X}](-d_{Z/Y})$,
where $\rho\colon T^*_Z(X\times Y)\to Z$ is the evident map. 
%This is a functorial assignment. 
%However, without further assumptions on $f$ and $g$, the cohomology sheaves of this complex are not necessarily coherent.
\begin{thm}[{\cite[Th\'eor\`eme 3.1.1]{La2}}]\label{laumoncorr}
Let $X\mapleft{f}Z\mapright{g}Y$ be a correspondence with $g$ proper. If $f$ is non-characteristic for $M\in D^b_m(X)$, then
\[ \tilde\gr \circ \Phi_{X|Y}(M) = \tilde{\Phi}_{X|Y}\circ \tilde\gr(M). \]
\end{thm}
\begin{remark}
Let $f\colon X\to Y$ be a morphism of smooth varieties. Let $\Gamma_f\subseteq X\times Y$ be the graph of $f$, $p_X\colon \Gamma_f\to X$ and $p_Y\colon \Gamma_f\to Y$ the projection maps. Then we have the correspondence $X\mapleft{p_X}\Gamma_f \mapright{p_Y} Y$, and $\Phi_{X|Y} = f_*$. If $f$ is proper, then Thm. \ref{laumoncorr} implies
$\tilde\gr f_* = f_{\pi*}f_d^!\tilde\gr [d_{X/Y}](-d_{X/Y})$.
Similarly, for the correspondence $Y\mapleft{p_Y}\Gamma_f\mapright{p_X}X$, one has $\Phi_{Y|X} = f^*$. If $f$ is smooth, then Thm. \ref{laumoncorr} implies
$\tilde \gr f^* = f_{d*}f_{\pi}^*\tilde\gr [d_{X/Y}]$.
Although we have obtained the above formulae as consequences of Thm. \ref{laumoncorr}, the proof of Thm. \ref{laumoncorr} proceeds by first obtaining these formulae. Further, in \cite{La1} and \cite{La2} the formulae do not keep track of the $\CC^*$-equivariant structure. Regardless, the (equivariant) formula for non-characteristic pullback is immediate from the definitions. The (equivariant) formula for proper pushforward requires a bit more work which is done in \cite[Lemma 2.3]{T}. With these in hand the proof of Thm. \ref{laumoncorr} proceeds exactly as that of \cite[Th\'eor\`eme 3.1.1]{La2}. We also note that \cite{La1} and \cite{La2} are written purely in the context of filtered $\cD$-modules. In this generality \cite[Th\'eor\`eme 3.1.1]{La2} does not quite hold - a crucial `strictness' assumption that is required for the formula for $\tilde\gr f_*$ is missing. However, this is not a problem for us, since if the filtered $\cD$-module structure is one underlying a mixed Hodge module, then this strictness assumption holds \cite[Th\'eor\`eme 1]{Sa88}.
\end{remark}
%\subsection{Non-characteristic pullback and $\tilde\gr$}%In this situation
%$f^*(M, F)[d_{X/Y}] = \cO_X\otimes_{f^{-1}\cO_Y} f^{-1}M$ with the filtered structure given by
%$F_i(\cO_X\otimes_{f^{-1}\cO_Y} f^{-1}M) = \cO_X\otimes_{f^{-1}\cO_Y} f^{-1}F_iM$.
%The $\cD_X$-module structure on $\cO_X\otimes_{f^{-1}\cO_Y} f^{-1} M$ can be described as follows. Let $\{x_i,\partial_i\}$ and $\{y_j, \partial_j\}$ be local coordinates for $X$ and $Y$ respectively. Then for $a\in \cO_X$ and $m\in M$,
%\[ \partial_i(a\otimes m) = \partial_i(a)\otimes m +\sum_j a\partial_i(f_j)\otimes \partial_jm,\]
%where $f_j$ is the $j$-th coordinate function of $f$.
%\begin{prop}\label{grpull}
%Let $(M,F)\in \MHM(Y)$.
%If $f$ is non-characteristic for $M$, then 
%\[\tilde\gr  f^*M [-d_{X/Y}]= f_{d*}f_{\pi}^* \tilde\gr M.\]
%\end{prop}
%
%\subsection{Proper pushforward and $\tilde\gr$}
%The following fundamental result is originally due to G. Laumon \cite{La1}, also see \cite[\S2.3]{Sa88}. For the equivariant version see \cite{T}.
%\begin{thm}[{\cite[\S5.6]{La1}, \cite[Lemma 2.3]{T}}]\label{grpush}If $f$ is proper, then
%\[\tilde \gr  f_* =  f_{\pi*}f_d^!\tilde\gr[d_{X/Y}](-d_{X/Y}).\]
% \qq^{d_{X/Y}}f_{\pi*}\t f^*(\tilde \gr (-) \otimes \pi_X^*\omega_{f}). \end{equation}
%\end{thm}

\subsection{The Steinberg variety}Let $\pi\colon \tilde\Nilcone \to G/B$ denote the cotangent bundle of $G/B$.   
Then $G\times \CC^*$ acts on $\tilde\Nilcone$ (the map $\pi$ is $G$-equivariant and $\CC^*$ acts via dilations of the fibres of $\pi$).
Under the isomorphism of \S\ref{ss:sign}, $\tilde\Nilcone \times \tilde\Nilcone$ is the cotangent bundle of $G/B\times G/B$. Further, $G\times \CC^*$ acts on $\tilde\Nilcone\times\tilde\Nilcone$ via the diagonal action.
%For each $w\in W$, write $Z_w$ for the conormal bundle to the $G$-orbit $Y_w\subseteq G/B\times G/B$.
The \emph{Steinberg variety} $Z\subseteq \tilde\Nilcone\times\tilde\Nilcone$ is defined by
\[ Z = \bigcup_{w\in W} T^*_{Y_w}(G/B\times G/B).\]
It is a closed $G\times \CC^*$ stable subvariety of $\tilde\Nilcone\times\tilde\Nilcone$. 
The projection $Z \to \tilde\Nilcone$ to either of the two factors is projective.
%
%\subsection{Another description of $Z$}We will need the following well known descriptions of $\tilde\Nilcone$ and $Z$. Let $\fg$ be the Lie algebra of $G$ and let $\Nilcone$ be the closed subvariety of nilpotent elements in $\fg$. Identify $G/B$ with the moduli space of Borel subalgebras of $\fg$. Then
%\[ \tilde\Nilcone = \{ (x,\fb)\in \Nilcone \times G/B \,|\, x\in \fb\},\]
%and $\pi\colon \tilde\Nilcone \to G/B$ is the restriction of the projection $\Nilcone\times G/B\to G/B$. Further,
%\[ Z = \{(x,\fb,-x,\fb')\in \tilde\Nilcone\times\tilde\Nilcone\,|\, x\in \fb\cap\fb'\}.\]
%The projection $Z\to \tilde\Nilcone$ (to either of the two factors) is projective, since it is the base change of the obvious projection map $\tilde\Nilcone \to \Nilcone$.
\subsection{Main player: coherent side}
Denote by $\tilde p_{13}\colon\tilde\Nilcone \times \tilde\Nilcone\times\tilde\Nilcone \to \tilde\Nilcone\times\tilde\Nilcone$ the projection on the first and third factor, and let $p_2\colon \tilde\Nilcone\times\tilde\Nilcone\times\tilde\Nilcone \to \tilde\Nilcone$ be projection on the second factor. Define $\tilde r\colon \tilde\Nilcone \times \tilde\Nilcone \times \tilde\Nilcone \to \tilde\Nilcone \times \tilde\Nilcone \times \tilde\Nilcone \times \tilde\Nilcone$ by $\tilde r = \id_{\tilde\Nilcone} \times \Delta \times \id_{\tilde\Nilcone}$, where $\Delta\colon \tilde\Nilcone \to \tilde\Nilcone \times \tilde\Nilcone$ is the diagonal embedding. Let
$\tilde\Hecke \subseteq D^{\CC^*}(\cO_{\tilde\Nilcone\times\tilde\Nilcone})$
be the full subcategory consisting of complexes whose cohomology sheaves are supported on $Z$.
Define a bifunctor $-\conv - \colon \tilde\Hecke \times \tilde\Hecke \to \tilde\Hecke$ by the formula
\[ M\conv N = \tilde p_{13*}\tilde r^* (M\boxtimes N).\]
%\[  M\conv N = \tilde p_{13*}(\tilde r^* (M\boxtimes N)\otimes p_2^*\pi^*\omega_{G/B}). \]
This endows $\tilde\Hecke$ with a monoidal structure. 
The unit is $\Delta_*\cO_{\Delta}$.
%The unit is $\Delta_*\pi^*\omega_{G/B}^{-1}$.

Define $\gamma\colon D^b(G\backslash(G/B\times G/B)) \to \tilde\Hecke$ by
\[ \gamma(M) = \tilde\gr \mathrm{For}(M) \otimes_{\cO_{\tilde \Nilcone\times\tilde\Nilcone}} \Delta_*\pi^*\omega_{G/B}^{-1} (-d_{G/B}). \]
%Define $\gamma\colon \Hecke \to \tilde\Hecke$ via the composition
%\[ \Hecke \mapright{\mathrm{For}} D^b_m(G/B\times G/B) \mapright{\qq^{d_{G/B}}\tilde\gr} D^{G\times \CC^*}(\cO_{\tilde\Nilcone \times \tilde\Nilcone}). \]
\begin{thm}\label{gammamonoidal}The functor
$\gamma$
is monoidal.
\end{thm}

\begin{proof}That $\gamma$ preserves the unit object can be seen directly. Now apply Thm. \ref{laumoncorr} to the correspondence
\[ G/B\times G/B\times G/B\times G/B \mapleft{r} G/B\times G/B\times G/B \mapright{p_{13}}G/B\times G/B.\]
%\[ 
%\xymatrix{
%&G/B\times G/B \times G/B\ar@/^0pc/[ld]_r \ar@/^0pc/[rd]^{p_{13}} & \\
%G/B\times G/B\times G/B\times G/B & & G/B\times G/B
%}\]
The $G$-equivariance of $M$ and $N$ implies that the characteristic variety of $M\boxtimes N$ is contained in $Z$. Further, $r$ is non-characteristic for $M\boxtimes N$. Consequently, $\gamma(M \conv N) = \gamma(M)\conv \gamma(N)$. To complete the proof we need to argue that the associativity constraints on both sides are compatible. The associativity constraint on either side is defined via the usual adjunction maps and base change (iso)morphisms. These are compatible with each other by \cite[\S2.6]{La2}.
\end{proof}
\begin{remark}\label{fairytale}Let me indicate how one \emph{might} remove the dependence on Hodge modules from our arguments. It is well known \cite{So} that $B$-equivariant hypercohomology defines a monoidal functor from (the non-mixed category) $D^b(B\backslash G/B)$ to (bi)modules for the equivariant cohomology ring $H^*_{B\times B}(\pt)$ that is full and faithful for morphisms between the $IC(X_w, \const{X}_w)$s. It is also known \cite{Sc} that $D^b(B\backslash G/B)$ is equivalent to the category $\Ext_B^{\bullet}(G/B)-\mathrm{dgperf}$ of perfect differential graded $\Ext_B^{\bullet}(G/B) = \End^{\bullet}_{D^b(B\backslash G/B)}(\bigoplus_{w\in W}IC(X_w, \const{X_w}))$-modules. Consequently, one obtains a full and faithful functor from $D^b(B\backslash G/B)$ to $\mathrm{Ho}(H^*_{B\times B}(\pt)-\mathrm{dgmod})$, the homotopy category of differential graded $H^*_{B\times B}(\pt)$-modules. Now the analogue of equivariant cohomology/pushing to a point on the coherent side should provide a similar functor to $\mathrm{Ho}(H^*_{B\times B}(\pt)-\mathrm{dgmod})$ allowing us to `match' our categories up. The slightly delicate issue here is that one should work with equivariant coherent sheaves on the Steinberg variety for the Langlands dual group, see \cite{Bez} and \cite{So}. Evidence that this idea has some hope of working is provided by the results of \cite{Bez}.
I hope to provide details and a precise formulation in future work.
\end{remark}

\end{document}